\documentclass{article}
\newcommand{\Addresses}{{
		\bigskip
		\footnotesize
		
		\textsc{Department of Mathematics, Technion - Israel Institute of Technology, Haifa, Israel}\par\nopagebreak
		\textit{E-mail address:} \texttt{ofir.gor@technion.ac.il}
}}

\author{Ofir Gorodetsky\thanks{This project has received funding from the European Research Council (ERC) under the European Union's Horizon 2020 research and innovation programme (grant agreement no.~851318) and from the Israel Science Foundation (grant no.~2088/24). O.G. is the Rabbi Dr.~Roger Herst Faculty Fellow at the Department of Mathematics, Technion. We thank Brad Rodgers for constructive discussions around Conjecture~\ref{conj}.}}
\title{High and odd moments in the Erd\H{o}s--Kac theorem}
\date{}
\usepackage[margin=1in]{geometry}
\usepackage{amssymb}
\usepackage{amsfonts}
\usepackage{amsthm}
\usepackage{amsmath}
\usepackage{mathtools}
\usepackage{hyperref}

\theoremstyle{plain}
\newtheorem{thm}{Theorem}[section]
\newtheorem{lem}[thm]{Lemma}  
\newtheorem{proposition}[thm]{Proposition}
\newtheorem{cor}[thm]{Corollary}

\newtheorem{claim}{Informal Claim}
\newtheorem{conj}{Conjecture}

\theoremstyle{remark}
\newtheorem*{remark*}{Remark}

\newcommand{\CC}{\mathbb{C}}
\newcommand{\PP}{\mathbb{P}}

\newcommand{\ZZ}{\mathbb{Z}}

\newcommand{\NN}{\mathbb{N}}
\newcommand{\RR}{\mathbb{R}}
\newcommand{\EE}{\mathbb{E}}
\newcommand*\diff{\mathop{}\!\mathrm{d}}
\newcommand{\AbsMom}{M}
\newcommand{\MMc}{a}
\newcommand{\Shift}{T}

\numberwithin{equation}{section}

\begin{document}

\maketitle
\begin{abstract}
Granville and Soundararajan showed that the $k$th moment in the Erd\H{o}s--Kac theorem is equal to the $k$th moment of the standard Gaussian distribution in the range $k=o((\log \log x)^{1/3})$, up to a negligible error term.
We show that their range is sharp: when $k/(\log \log x)^{1/3}$ tends to infinity, a different behavior emerges, and odd moments start exhibiting similar growth to even moments. For odd $k$ we find the asymptotics of the $k$th moment when $k=O((\log \log x)^{1/3})$, where previously only an upper bound was known. Our methods are flexible and apply to other distributions, including the Poisson distribution, whose centered moments turn out to be excellent approximations for the Erd\H{o}s--Kac moments.
\end{abstract}

\section{Introduction}
Let $\omega(n)=\sum_{p \mid n}1$ be the prime-divisor function. The celebrated Erd\H{o}s--Kac theorem \cite{Erdos} states that 
\begin{equation}\label{eq:ek} 
\frac{\omega(n)-\log \log n}{\sqrt{\log \log n}}      \xrightarrow[x \to \infty]{d}  G
\end{equation}
where $n$ is sampled uniformly at random from the positive integers up to $x$, and $G \sim N(0,1)$. The arrow indicates convergence in distribution. We may replace $\log \log n$ in \eqref{eq:ek} with $\log \log x$ because $(\log \log n - \log \log x)/\sqrt{\log \log x}$ tends to $0$ in probability.
The moments of the standard Gaussian distribution determine it uniquely, so to prove \eqref{eq:ek} it suffices to show that
\begin{equation}\label{eq:mom} \lim_{x \to \infty}\EE_{n \le x}\bigg(\frac{\omega(n)-\log \log x}{\sqrt{\log \log x}} \bigg)^k = \EE (G^k)
\end{equation}
holds for all $k \in \NN$. The case $k=1$ of \eqref{eq:mom} follows from a classical result of Mertens on the asymptotics of $\sum_{p \le x} 1/p$ \cite[Thm.~2.7]{MV}. The case $k=2$ is due to Tur\'{a}n \cite{Turan} and implies the Hardy--Ramanujan Theorem on the normal order of $\omega$. Delange \cite{Delange1953} proved \eqref{eq:mom} for all $k \in \NN$ (see \cite{Halberstam,Halberstam1956,Delange,DH} for generalizations to other additive functions).
Much later, Granville and Soundararajan used a sieve-inspired approach to study \eqref{eq:mom} with $k$ growing \cite[Thm.~1]{Granville2007}. They proved that
\begin{equation}\label{eq:GS}
\EE_{n \le x}\bigg(\frac{\omega(n)-\log \log x}{\sqrt{\log \log x}} \bigg)^k = \EE (G^k) + O\bigg( \EE 
(|G|^k) \frac{k^{3/2}}{\sqrt{\log \log x} }\bigg)
\end{equation}
holds uniformly for $k \le (\log \log x)^{1/3}$.
Delange  \cite{Delange1959} used the moment generating function of $\frac{\omega(n)-\log \log x}{\sqrt{\log \log x}}$ to show the following. For $k \in \NN$, 
\begin{equation}\label{eq:D}
	\EE_{n \le x}\bigg(\frac{\omega(n)-\log \log x}{\sqrt{\log \log x}} \bigg)^k =\frac{k!A_k(x)}{(\log \log x)^{k/2}}+O_k\bigg(\frac{1}{\log x}\bigg)
\end{equation}
holds where $A_k(x)$ is the coefficient of $z^k$ in the Maclaurin expansion of $F(e^z)(\log x)^{e^z-z-1}$ and $F$ is the entire function
\begin{equation}\label{eq:Fdef} F(z):=\frac{1}{\Gamma(z)}\prod_{p}\bigg(1-\frac{1}{p}\bigg)^z \bigg(1+\frac{z}{p-1}\bigg).
\end{equation}
He used \eqref{eq:D} to give a new proof of \eqref{eq:mom}. Ghosh \cite{Ghosh} estimated the $r$th absolute moment of $(\omega(n)-\log \log x)/\sqrt{\log \log x}$, $r$ being any real, positive number, and showed it tends to $\EE(|G|^r)$.
\subsection{Results}
We extend and improve \eqref{eq:GS} in terms of range and error term. We compute the asymptotics of the $k$th centered moment of $\omega$ in the range $k=O(\log \log x)$, and in particular for odd $k$.  This requires a genuine modification of the Sath\'e--Selberg method \cite{Sathe,Selberg}. Namely, we apply a suitable saddle point analysis to extract information about the moments from their generating function.

 While the formulas for the moments ae complicated, we can state the following clean theorem, comparing moments of $\omega$ with moments of a Poisson random variable, as well as with the moments of sums of Bernoulli random variables. Recall the Meissel--Mertens constant
	\begin{equation}\label{eq:bdef}
	\MMc = \gamma+\sum_{p} \left(\log\left(1-\frac{1}{p}\right)+\frac{1}{p}\right)>0
\end{equation}
which arises as the limit of $\sum_{p\le x} 1/p -\log \log x$ \cite[Thm.~2.7]{MV} and coincides with $F'(1)$ for $F$ in \eqref{eq:Fdef}.
\begin{thm}\label{thm:main}
Fix $A>0$. Suppose $x\ge 3$ and $2 \le k \le A\log \log x$. Define $r>0$ via 
\[ r(e^r-1)=  \frac{k}{\log \log x}.\] 
\begin{itemize} 
\item 
Let $X(\log \log x)\sim \mathrm{Poisson}(\log \log x)$. As $x\to \infty$,
\[\EE_{n \le x}(\omega(n)-\log \log x-\MMc)^k \sim  \frac{F(e^r)}{e^{r\MMc}} \EE(X(\log \log x)-\log \log x)^k\]
where $F$ is given in \eqref{eq:Fdef} and $\MMc$ is defined in \eqref{eq:bdef}.
\item Let $(X_p)_{p \le x}$ be independent random variables with $X_p \sim \mathrm{Bernoulli}(1/p)$. As $x\to \infty$,
\[	\EE_{n \le x}\bigg(\omega(n)-\sum_{p\le x} \frac{1}{p}\bigg)^k \sim  \frac{e^{-\gamma (e^r-1)}}{\Gamma(e^r)} \EE\bigg(\sum_{p \le x} X_p - \sum_{p\le x} \frac{1}{p}\bigg)^k.\]
\end{itemize}
If $k=o(\log \log x)$ then $r=o(1)$ and so $F(e^r)/e^{r\MMc}\sim e^{-\gamma (e^r-1)}/\Gamma(e^r) \sim 1$.
\end{thm}
The proof of Theorem~\ref{thm:main} proceeds by computing separately the asymptotics of the moments in the right- and left-hand sides, and then comparing. This is in the spirit of the work of Radziwi\l\l, comparing large deviations of $\omega(n)-\sum_{p \le x} 1/p$ with those of $\mathrm{Poisson}(\log \log x)-\log \log x$ \cite[Cor.~3]{radziwill2011structure} and of sums of centered Bernoulli-s \cite[Eq.~(1.5)]{radziwill2009large}.

Next we describe the behavior of  centered moments of $\omega$ in the range $k=o(\sqrt{\log \log x})$, including the behavior of odd moments. For $k \in \NN$ we shall use the following notation:
\[\mu_k:= \EE (G^k) = \begin{cases} \frac{k!}{2^{k/2}(k/2)!} & \text{if }k\text{ is even,} \\ 0 &\text{ otherwise,}\end{cases}\qquad  \AbsMom_k:=\frac{k!}{2^{k/2}\Gamma\left(\frac{k}{2}+1\right)} \asymp \EE 
(|G|^k).\]
\begin{thm}\label{thm:lowk}
Fix $A> 0$. Suppose $x\ge 3$. For $1\le k\le A(\log \log x)^{1/3})$ we have
\[\EE_{n \le x}\bigg(\frac{\omega(n)-\log \log x}{\sqrt{\log \log x}} \bigg)^k =\bigg(\mu_k +\bigg(\frac{k(k-1)}{6}+k\MMc\bigg)\frac{\mu_{k-1}}{\sqrt{\log \log x}}\bigg) \bigg(1+O_A\bigg( \frac{k^3}{\log \log x}\bigg)\bigg)\]
where $\MMc$ is as in \eqref{eq:bdef}.	For  $k\asymp (\log \log x)^{1/3}$ we have
	\begin{align*}
			\AbsMom_k^{-1} \EE_{n \le x}&\bigg(\frac{\omega(n)-\log \log x}{\sqrt{\log \log x}} \bigg)^k \\
			& \sim \frac{1}{2}\exp\bigg( \frac{k^{3/2}}{6(\log \log x)^{1/2}}\bigg) \bigg(1 +(-1)^k \exp\left( - \frac{k^{3/2}}{3(\log \log x)^{1/2}}\right)\bigg) \asymp 1
	\end{align*}
	as $x \to \infty$. If $k/(\log \log x)^{1/3} \to \infty$ while $k=o(\sqrt{\log \log x})$ then letting  $r=\sqrt{k/\log \log x}$ we have
\[		\AbsMom_k^{-1} \EE_{n \le x}\bigg(\frac{\omega(n)-\log \log x}{\sqrt{\log \log x}} \bigg)^k \sim \frac{1}{2} (\log x)^{e^{r}-(1+r+r^2/2)}\to \infty.\]
\end{thm}
In the next section we shall state our results in full, which include the behavior of the moments of $\omega$ for $k=O(\log \log x)$. Our next two theorems may be of interest to probability theorists. We have an analogue of Theorem~\ref{thm:lowk} for the Poisson distribution:
\begin{thm}\label{thm:lowkpo}
Let $\lambda >0$ and $X(\lambda)\sim \mathrm{Poisson}(\lambda)$. Fix $A>0$. For $1\le k \le A\lambda^{1/3}$,
	\begin{equation}\label{eq:ratio3po}
	\EE\bigg(\frac{X(\lambda)-\lambda}{\sqrt{\lambda}} \bigg)^k = \bigg( \mu_k + \frac{k(k-1)\mu_{k-1}}{6\sqrt{\lambda}}\bigg)\bigg(1+O_A\bigg( \frac{k^3}{\lambda}\bigg)\bigg).
	\end{equation}
For $k\asymp \lambda^{1/3}\to \infty$,
	\begin{equation}\label{eq:ratio2po}
			\AbsMom_k^{-1} \EE\bigg(\frac{X(\lambda)-\lambda}{\sqrt{\lambda}} \bigg)^k \sim \frac{1}{2}\exp\bigg( \frac{k^{3/2}}{6\lambda^{1/2}}\bigg) \bigg(1 +(-1)^k \exp\left( - \frac{k^{3/2}}{3\lambda^{1/2}}\right)\bigg) \asymp 1.
	\end{equation}
If $k/\lambda^{1/3} \to \infty$ while $k=o(\lambda^{1/2})$ then letting  $r=\sqrt{k/\lambda}$ we have
	\begin{equation}\label{eq:ratiopoisson}
		\AbsMom_k^{-1} \EE\bigg(\frac{X(\lambda)-\lambda}{\sqrt{\lambda}} \bigg)^k \sim \frac{1}{2} e^{\lambda(e^{r}-(1+r+r^2/2))}\to \infty.
	\end{equation}
\end{thm}
We have a result for the \textit{Poisson Binomial distribution}, i.e.~the sum of independent Bernoulli-s.
\begin{thm}\label{thm:lowkpobern}
Let $(p_i)_{i=1}^{n}$ be  reals in $[0,1]$, $(Y_i)_{i=1}^{n}$ be independent random variables with $Y_i \sim \mathrm{Bernoulli}(p_i)$ and $Y=\sum_{i=1}^{n}Y_i$. Denote  $\lambda=\sum_{i=1}^{n} p_i$ and let  $X(\lambda) \sim \mathrm{Poisson}(\lambda)$. For $2 \le k =o(\lambda/\max\{1,\sum_{i=1}^{n}p_i^2\}$),
\[ \EE(Y-\lambda)^k \sim \EE (X(\lambda)-\lambda)^k.\]
\end{thm}
We mention Le Cam's inequality \cite{LeCam}, or rather its later refinement \cite[Eq.~(5.5)]{Steele}, as it is in the spirit of Theorem~\ref{thm:lowkpobern}. It says, in the notation of Theorem~\ref{thm:lowkpobern}, that
\[ \mathrm{d}_{\mathrm{TV}}(Y,X(\lambda))\le \min\{1,\lambda^{-1}\}\sum_{i=1}^{n}p_i^2. \]
\subsection{Remark on primes in short intervals}
Let $\lambda>0$. Under the assumption of the Hardy--Littlewood $k$-tuple conjecture, Gallagher \cite{Gallagher} proved that the random variable
\[ \sum_{t\le p  < t+\lambda \log x} 1,\]
where $t$ is chosen uniformly at random from $[0,x]$, tends in distribution to $X(\lambda)\sim \mathrm{Poisson}(\lambda)$ as $x\to \infty$. He proved this by showing, for each $k \in \NN$, that 
\begin{equation}\label{eq:mompo}
\lim_{x\to \infty} \EE_{t\le x} \bigg(\sum_{t\le p  < t+\lambda \log x} 1\bigg)^k =\EE X^k(\lambda).
\end{equation}
We put forth a conjecture which generalizes \eqref{eq:mompo}, motivated by Theorem~\ref{thm:main}.
\begin{conj}\label{conj}
	Fix $k\ge 2$. If $0< h=x^{o(1)}$ then
	\[ \EE_{t\le x} \bigg(\sum_{t\le p< t+h}1 - \int_{t}^{t+h}\frac{\diff{t}}{\log t}\bigg)^k\sim \EE_{t\le x} \bigg( \frac{\sum_{t\le p< t+h}\log p - h}{\log x}\bigg)^k \sim \EE \bigg(X\bigg(\frac{h}{\log x}\bigg)-\frac{h}{\log x}\bigg)^k \]
as $x\to\infty$, where	$X(h/\log x)\sim \mathrm{Poisson}(h/\log x)$.
\end{conj}
When $h=o(\log x)$ and $k \ge 2$, $\EE (X(h/\log x)-h/\log x)^k \sim h/\log x$ and Conjecture~\ref{conj} is straightforward to verify using an upper bound sieve. For $h \asymp \log x$, Conjecture~\ref{conj} is \textit{equivalent} to Gallagher's \eqref{eq:mompo}. Now suppose $h/\log x \to \infty$ and $h=x^{o(1)}$. If $k$ is even, \cite[Conj.~1]{MS} says that $\EE_{t\le x} ( \sum_{t\le p< t+h}\log p - h)^k \sim \mu_{k} (h\log x)^{k/2}$. This is consistent with Conjecture~\ref{conj} since $\EE (X(h/\log x)-h/\log x)^k \sim (h/\log x)^{k/2} \EE(G^k)$. If $k$ is odd, \cite[Conj.~1]{MS} only says that $\EE_{t\le x} ( \sum_{t\le p< t+h}\log p - h)^k = o((h\log x)^{k/2})$. In contrast, Conjecture~\ref{conj} says that when $k$ is odd, the missing lower order term in \cite[Conj.~1]{MS} is explained by the Poisson distribution. Since $\EE (X(h/\log x)-h/\log x)^k \sim (h/\log x)^{(k-1)/2}\mu_{k+1} (k-1)/6$ for odd $k\ge 3$, Conjecture~\ref{conj} is also consistent with the conditional estimate in \cite[Thm~1.4]{LB}.
\begin{remark*}
Let $Y(n,p)\sim \mathrm{Binomial}(n,p)$. As $x\to \infty$, one has $\EE (X(\tfrac{h}{\log x})-\tfrac{h}{\log x})^k\sim \EE (Y(\lfloor h\rfloor,\tfrac{1}{\log x})-\tfrac{h}{\log x})^k$ as long as $h\to \infty$ or $h \in \NN$. Here $k\ge 2$ is fixed.
\end{remark*}
\subsection*{Structure of paper}
In \S\ref{sec:full} we state our full results on centered moments, from which Theorems~\ref{thm:main}--\ref{thm:lowkpobern} follow quickly, see \S\ref{sec:detailsthms}. In \S\ref{sec:lemmas} we state our technical lemmas on generating functions, and use them to deduce the results in \S\ref{sec:full}. In \S\ref{sec:lemmasproofs} we prove the technical lemmas stated in \S\ref{sec:lemmas}. In \S\ref{sec:intuition} we explain \textit{why} a transition occurs in the $k$th moment of $\omega(n)-\log \log n$ when $k\asymp (\log \log x)^{1/3}$.
\section{Full results}\label{sec:full}
The following result is a strengthening of \eqref{eq:GS}, but for the Poisson distribution instead of $\omega$.
\begin{proposition}\label{prop:babypoisson}
Let $\lambda >0$. Let $X(\lambda) \sim \mathrm{Poisson}(\lambda)$.	Fix $A>0$. For $1 \le k \le A \lambda^{1/3}$ we have
\[	\EE\bigg(\frac{X(\lambda)-\lambda}{\sqrt{\lambda}} \bigg)^k = \bigg( \mu_k + \frac{k(k-1)\mu_{k-1}}{6\sqrt{\lambda}}\bigg)\bigg(1+O_A\bigg( \frac{k^3}{\lambda}\bigg)\bigg).\]
The implied constant depends only on $A$.
\end{proposition}
The following result goes beyond the range of Proposition~\ref{prop:babypoisson}.
\begin{proposition}\label{prop:po}
	Let $\lambda >0$. Let $X(\lambda) \sim \mathrm{Poisson}(\lambda)$.	Fix $A>0$. For $1 \le k \le A\lambda$ we have
	\begin{align*}
		\EE\bigg(\frac{X(\lambda)-\lambda}{\sqrt{\lambda}} \bigg)^k &=\frac{\AbsMom_k}{2}\exp(\lambda ( e^r-(1+r+r^2/2)))\\
		&\qquad \cdot ( 1 + (-1)^k \exp( \lambda (e^{-r}-e^{r}+2r))(1+O(1/k))+ O_A(k^2/\lambda + r))
	\end{align*}
	where $r=\sqrt{k/\lambda}$.  The first implied constant is absolute, the second one depends only on $A$.
\end{proposition}
Proposition~\ref{prop:po} yields an asymptotic result when $k=o(\sqrt{\lambda})$, and otherwise only implies an upper bound. The following result does yield an asymptotic result for $k=O(\lambda)$, as long as $k/\lambda^{1/3}\to \infty$.
\begin{proposition}\label{prop:po2}
Let $\lambda>0$. Let $X(\lambda) \sim \mathrm{Poisson}(\lambda)$. Fix $A>0$. Given $k \ge 1$ let $r$ be the positive solution to $r(e^r-1)=k/\lambda$.  For $1 \le k \le \lambda A$ we have $r \asymp_A\sqrt{k/\lambda}$ and
	\[\EE\bigg(\frac{X(\lambda)-\lambda}{\sqrt{\lambda}} \bigg)^k =\frac{\AbsMom_k}{2} \frac{\exp(kS)}{\sqrt{s}} \left(1 +O_A(1/k + e^{-c_A k^{3/2}/\lambda^{1/2}})\right)\]
	where $c_A$ is a positive constant that depends only on $A$, and $S$ and $s$ are functions of $r$ defined as in Lemma~\ref{lem:technical2}.
\end{proposition}
The following comparative result for Poisson moments will be needed.
\begin{proposition}\label{prop:comp}
Let $\lambda>0$. Fix $A>0$.  Let $\lambda'>0$ be a quantity satisfying $-A\le \lambda'-\lambda \le A$.  Let $X(\lambda) \sim \mathrm{Poisson}(\lambda)$ and $X(\lambda') \sim \mathrm{Poisson}(\lambda')$. Given $k \ge 2$ let $r$ be the positive solution to $r(e^r-1)=k/\lambda$.  If $k \le A \lambda$  then, as $\lambda \to \infty$,
\[\EE(X(\lambda)-\lambda )^k \sim e^{(\lambda-\lambda')(e^r-r-1)} \EE(X(\lambda')-\lambda' )^k.\]
\end{proposition}
The following result strengthens \eqref{eq:GS} and parallels Proposition~\ref{prop:babypoisson}.
\begin{proposition}\label{prop:omega}
	Fix $A>0$. For $x\ge 3$, $1 \le k \le A(\log \log x)^{1/3}$ and $-A \le \Shift \le A$ we have
\[			\EE_{n \le x}\bigg(\frac{\omega(n)-\log \log x-\Shift}{\sqrt{\log \log x}} \bigg)^k =\mu_k +\bigg(\frac{k(k-1)}{6}+k(\MMc-\Shift)\bigg) \frac{\mu_{k-1}}{\sqrt{\log \log x}} + O_A\bigg(\AbsMom_k \bigg(\frac{k^{3}}{\log \log x}\bigg)^{1+\frac{\mathbf{1}_{2\nmid k}}{2}}\bigg)\]
where $\MMc$ is as in \eqref{eq:bdef}. The implied constant depends only on $A$.
\end{proposition}
The next two results parallel Propositions~\ref{prop:po} and \ref{prop:po2}, respectively.
\begin{proposition}\label{prop:kac}
Fix $A>0$. For $x \ge 3$, $1 \le k \le A \log \log x$ and $-A\le \Shift \le A$ we have
\begin{align*}
		&\EE_{n \le x}\bigg(\frac{\omega(n)-\log \log x-\Shift}{\sqrt{\log \log x}} \bigg)^k =\frac{\AbsMom_k}{2}(\log x)^{e^r-(1+r+r^2/2)} \\
		&\qquad\qquad \cdot \bigg( \frac{F(e^r)}{e^{r\Shift}} + (-1)^k \frac{F(e^{-r})}{e^{-r\Shift}}(\log x)^{e^{-r}-e^{r}+2r}(1+O(1/k))+ O_A\bigg(\frac{k^2}{\log \log x}+r\bigg)\bigg)
	\end{align*}
where $r=\sqrt{k/\log \log x}$. The first implied constant is absolute, the second one depends only on $A$.
\end{proposition}
\begin{proposition}\label{prop:kac2}
	Given $k \ge 1$ and $x \ge 3$ let $r$ be the positive solution to $r(e^r-1) = k/\log \log x$. Fix $A>0$. For $1 \le k \le A \log \log x$ and $-A\le \Shift \le A$ we have $r \asymp_A\sqrt{k/\log \log x}$ and 		
\[	\EE_{n \le x}\bigg(\frac{\omega(n)-\log \log x-\Shift}{\sqrt{\log \log x}} \bigg)^k =\frac{\AbsMom_k}{2} \frac{\exp(kS)}{\sqrt{s}}\frac{F(e^r)}{e^{r\Shift}}  \left(1 +O_A(1/k+e^{-c_A k^{3/2}/\sqrt{\log \log x}})\right)\]
	where $c_A$ is a positive constant that depends only on $A$, and $S$ and $s$ are functions of $r$ defined as in Lemma~\ref{lem:technical2}. The implied constant depends only on $A$. 
\end{proposition}
We present an analogue of Propositions~\ref{prop:omega}--\ref{prop:kac2} for the Poisson binomial distribution.
\begin{proposition}\label{prop:pb}
	Let $(p_i)_{i=1}^{n}$ be  reals in $[0,1]$ and $(Y_i)_{i=1}^{n}$ be independent random variables with $Y_i \sim \mathrm{Bernoulli}(p_i)$. Denote $Y=\sum_{i=1}^{n}Y_i$,  $\lambda=\sum_{i=1}^{n} p_i$ and $\Lambda = \sum_{i=1} p_i^2$. Fix $A>0$. Suppose $1\le k \le A \lambda/\Lambda$.
	\begin{enumerate}
		\item If $1 \le k \le A\lambda^{1/3}$ then
		\begin{equation}\label{eq:binompoisson}	\EE\bigg(\frac{Y-\lambda}{\sqrt{\lambda}} \bigg)^k = \bigg(\mu_k + \frac{k(k-1)\mu_{k-1}}{6\sqrt{\lambda}}\bigg) \bigg(1+O_A\bigg( \frac{k^3}{\lambda}+\frac{k\Lambda}{\lambda}\bigg)\bigg).
			\end{equation}
		\item If $1 \le k \le A\lambda$ and we let $r=\sqrt{k/\lambda}$, then
		\begin{multline*}
		\EE\bigg(\frac{Y-\lambda}{\sqrt{\lambda}} \bigg)^k =\frac{\AbsMom_k}{2}\exp(\lambda ( e^r-(1+r+r^2/2))) \cdot \bigg( \prod_{i=1}^{n}\frac{1+p_i(e^r-1)}{e^{p_i(e^r-1)}}  \\+(-1)^k \prod_{i=1}^{n}\frac{1+p_i(e^{-r}-1)}{e^{p_i(e^{-r}-1)}} \exp( \lambda (e^{-r}-e^{r}+2r))(1+O(1/k))+ O_A((k^2+\Lambda)/\lambda + r)\bigg).
	\end{multline*}
	\item If $1 \le k \le A\lambda$ and instead we define $r>0$  via $r(e^r-1)=k/\lambda$, then
	\[\EE\bigg(\frac{Y-\lambda}{\sqrt{\lambda}} \bigg)^k =\frac{\AbsMom_k}{2} \frac{\exp(kS)}{\sqrt{s}}\prod_{i=1}^{n}\frac{1+p_i(e^r-1)}{e^{p_i(e^r-1)}} \left(1 +O_A(1/k + e^{-c_A k^{3/2}/\lambda^{1/2}})\right)\]
	where $c_A$ is a positive constant that depends only on $A$, and $S$ and $s$ are functions of $r$ defined as in Lemma~\ref{lem:technical2}.	
	\end{enumerate}
\end{proposition}
For fixed $k$, \eqref{eq:binompoisson} can be explained combinatorially: by the multinomial theorem and the independence of the $Y_i$-s,
\[ \EE (Y-\lambda)^k= \sum_{a_1+\ldots+a_n=k} \binom{k}{a_1,\ldots,a_n} \prod_{i=1}^{n} \EE (Y_i-p_i)^{a_i}.\]
If $k$ is even then the main contribution comes from $k/2$ of the $a_i$-s being equal to $2$ and the rest being $0$. If $k$ is odd the main contribution comes from one $a_i$ being equal to $3$, and the rest being equal to $2$ or $0$.
\subsection{Proofs of theorems}\label{sec:detailsthms}
\subsubsection{Proof of Theorem~\ref{thm:main}}
To prove the first part we compare Propositions~\ref{prop:omega}--\ref{prop:kac2} (with $\Shift =\MMc$) to Propositions~\ref{prop:babypoisson}--\ref{prop:po2} (with $\lambda=\log \log x$). To prove the second part we consider two cases. If $k=O((\log \log x)^{1/3})$ we compare Propositions~\ref{prop:omega}--\ref{prop:kac} (with $\Shift=\sum_{p\le x}1/p - \log \log x =\MMc + o(1)$) to Proposition~\ref{prop:pb} (with $(p_i)_{i\le n}=(1/p)_{p\le x}$). If $k/(\log \log x)^{1/3}\to \infty$ we make a chain of comparisons: we compare Proposition~\ref{prop:kac2} (with $\Shift=\sum_{p\le x}1/p-\log \log x$) to Proposition~\ref{prop:po2} (with $\lambda = \log \log x$), we then apply Proposition~\ref{prop:comp} (with $\lambda = \sum_{p\le x}1/p$ and $\lambda'=\log \log x$), and finally we compare Proposition~\ref{prop:po2} (with $\lambda=\sum_{p \le x} 1/p$) to the last part of Proposition~\ref{prop:pb} (with $(p_i)_{i\le n}=(1/p)_{p\le x}$).
\subsubsection{Proof of Theorem~\ref{thm:lowkpo}}
The estimate \eqref{eq:ratio3po} is Proposition~\ref{prop:babypoisson}. To deduce \eqref{eq:ratio2po}--\eqref{eq:ratiopoisson} we simplify Proposition~\ref{prop:po} as follows. Since $r=\sqrt{k/\lambda} \le \sqrt{A}$ by assumption, we have
\begin{equation}\label{eq:tayapp}
	e^{r}-(1+ r+r^2/2) \in [c_A r^3,C_A r^3], \quad e^{-r}-e^r+2r \in [-C_A r^3,-c_A r^3]
\end{equation}
for some constants $C_A>c_A>0$, which implies that \eqref{eq:ratiopoisson} holds if $k/\lambda^{1/3}\to \infty$ while $k=o(\lambda^{1/2})$. Furthermore, since $e^{\pm r}=1\pm r+r^2/2\pm r^3/6+O_A(r^4)$, we deduce from Proposition~\ref{prop:po} that \eqref{eq:ratio2po} holds for $k\asymp \lambda^{1/3}\to\infty$.
\subsubsection{Proof of Theorem~\ref{thm:lowk}}
We simplify Propositions~\ref{prop:omega} and \ref{prop:kac} (with $\Shift=0$) in the same way we deduced Theorem~\ref{thm:lowkpo}, i.e.~we apply \eqref{eq:tayapp}, $e^{\pm r}=1\pm r+r^2/2\pm r^3/6+O_A(r^4)$ and also $F(e^{\pm r}) = 1+O_A(r)$ whenever $r=\sqrt{k/\log \log x} \le \sqrt{A}$. 
\subsubsection{Proof of Theorem~\ref{thm:lowkpobern}}
We compare Proposition~\ref{prop:pb} to Propositions~\ref{prop:babypoisson}--\ref{prop:po2}.
\section{Generating function lemmas}\label{sec:lemmas}
The following easy lemma will be used to handle some error terms.
\begin{lem}\label{lem:tech0}
	Let $H(z) = \sum_{k \ge 0} a_k z^k$ be a power series with $a_k \in \CC$. Given $\lambda >0$ we write $H$ as \[ H(z) =\exp( \lambda z^2/2)H_1(z).\]
	Given $k \ge 0$ we let $r=\sqrt{k/\lambda}$. We have $k! a_k \ll \AbsMom_k \lambda^{k/2} \max_{|z|=r}|H_1(z)+(-1)^kH_1(-z)|$ if $H$ converges absolutely for $|z| = r$. The implied constant is absolute.
\end{lem}
The following is a quick consequence of Lemma~\ref{lem:tech0}.
\begin{cor}\label{cor:tech}
		Let $H(z) = \sum_{k \ge 0} a_k z^k$ be a power series with $a_k \in \CC$. Given $\lambda >0$ we write $H$ as \[ H(z) =\exp( \lambda z^2/2)H_1(z).\]
	Write $H_1(z)$ as $H_1(z) = \sum_{j \ge 0} b_j z^j$. Given $k\ge 1$ and $0 \le m\le k$ we have
	\[ k! a_k = \lambda^{\frac{k}{2}}\bigg(\sum_{j=0}^{m}b_j \frac{k!}{(k-j)!} \lambda^{-\frac{j}{2}}\mu_{k-j}+ O(E)\bigg)\]
	with an implied absolute constant, where $E:=0$ if $m=k$ and otherwise 
	\[ E := \frac{k!}{(k-m-1)!}  \lambda^{-\frac{m+1}{2}}\AbsMom_{k-m-1} \frac{\max_{|z|=\sqrt{(k-m-1)/\lambda}} |H^{(m+1)}_1(z)+(-1)^k H^{(m+1)}_1(-z)|}{(m+1)!}\]
	as long as $H$ converges absolutely for $|z|= \sqrt{(k-m-1)/\lambda}$.
\end{cor}
Corollary~\ref{cor:tech} is sufficiently strong to recover \eqref{eq:GS} and yield Propositions~\ref{prop:babypoisson} and \ref{prop:omega}, and \eqref{eq:binompoisson}. The following two lemmas are needed for going beyond the range of \eqref{eq:GS}.
\begin{lem}\label{lem:tech1}
			Let $H(z) = \sum_{k \ge 0} a_k z^k$ be a power series with $a_k \in \CC$. Given $\lambda >0$ we write $H$ as \[ H(z) =\exp( \lambda z^2/2)H_1(z).\]
Given $k\ge 1$ we let $r=\sqrt{k/\lambda}$. We have
	\[ k! a_k = \AbsMom_k\lambda^{k/2}\left( \mathbf{1}_{2 \mid k} \frac{H_1(r) +H_1(-r)}{2} + \mathbf{1}_{2 \nmid k} \frac{H_1(r)-H_1(-r)}{2}(1+O(1/k)) + O( E/k)\right)\]
	as long as $H$ converges absolutely for $|z|= r$, where
	\begin{equation}\label{eq:Edef}
		E :=  r\max_{|z|= r}|H_1'(z)|+ r^2\max_{|z|= r} |H''_1(z)|,
	\end{equation}
 and all implied constants are absolute. 
\end{lem}

\begin{lem}\label{lem:technical2}
	Fix $A>0$. Let $H(z) = \sum_{k \ge 0} a_k z^k$ be a power series converging for $|z|\le A$. Given $\lambda >0$ and a power series $H_1(z)$ converging in $|z| \le A$, we write $H$ as \[ H(z) =\exp( \lambda(e^z-z-1))(H_1(z)+H_2(z)).\]
	Given $k \ge 1$ we let $r$ be the positive solution to $r(e^r-1)=k/\lambda$. For $1 \le k \le \lambda A(e^A-1)$ we have $r \asymp_A\sqrt{k/\lambda}$ and
	\[ k! a_k = \lambda^{k/2}\frac{\AbsMom_k}{2} \frac{\exp(kS)}{\sqrt{s}} ( H_1(r) + O_A ( E + \max_{|z|= r}|H_2(z)| ))\]
	where
\[E:=(|H_1(r)|+ r\max_{|z|= r}|H_1'(z)|+ r^2\max_{|z|= r} |H''_1(z)|) \cdot( k^{-1} + \exp(-c_A k^{3/2}\lambda^{-1/2}))\]
and $S=S_1+S_2$, $s=bt$ and
		\begin{align*}
		b &:=\frac{1}{2}\left( e^r + \frac{e^r-1}{r}\right) = 1 + O_A(r),\qquad t:=\frac{r}{e^r-1} =1+O_A(r),\\
		S_1 &: =  t  \frac{e^r-\frac{r^2}{2}-r-1}{r^2} = \frac{ r}{6}(1+O_A(r))  \asymp_A (k/\lambda)^{1/2},\\
		S_2 &:= \frac{1}{2}(t-1 - \log t)= \frac{r^2}{16}(1+O_A(r))\asymp_A k/\lambda
	\end{align*}
	and the implied constants and $c_A>0$ depend only on $A$.
\end{lem}

\subsection{Proof of Propositions~\ref{prop:babypoisson}--\ref{prop:comp}}
Let $\lambda>0$ and $X(\lambda)\sim \mathrm{Poisson}(\lambda)$. The moment generating function of $X(\lambda)-\lambda$ is 
\[ \sum_{k \ge 0 } \frac{z^k}{k!} \EE (X(\lambda)-\lambda)^k =  \EE e^{ z (X(\lambda)-\lambda)} =\sum_{i \ge 0} \lambda^i \frac{e^{-\lambda}}{i!} e^{z(i-\lambda)} = e^{ \lambda(e^z-z-1)}.\]
Corollary~\ref{cor:tech} and Lemmas~\ref{lem:tech1}-\ref{lem:technical2} estimate the $k$th coefficient of the above series, multiplied by $k!$ (this is exactly the $k$th moment):
\begin{itemize}
	\item To prove Proposition~\ref{prop:babypoisson} we apply Corollary~\ref{cor:tech} with $H(z)=e^{\lambda(e^z-z-1)}$ and  $m=1$ (if $k$ is even) or $m=3$ (if $k$ is odd).	Under this choice of $H$, $H_1(z)=e^{\lambda(e^z-z^2/2-z-1)}=1+\lambda z^3/3!+\ldots$.
	\item To prove Proposition~\ref{prop:po} we apply Lemma~\ref{lem:tech1} with $H(z)=e^{\lambda(e^z-z-1)}$.
	\item To prove Proposition~\ref{prop:po2} we apply Lemma~\ref{lem:technical2} with  $H(z)=e^{\lambda(e^z-z-1)}$, $H_1\equiv 1$ and $H_2\equiv 0$.
	\item To prove Proposition~\ref{prop:comp} we consider different ranges. If $k=o(\lambda^{1/3})$, we apply Proposition~\ref{prop:babypoisson} twice, with $\lambda$ and $\lambda'$, and compare. If $k \asymp \lambda^{1/3}$, we apply Proposition~\ref{prop:po} twice, with $\lambda$ and $\lambda'$, and compare. If $k /\lambda^{1/3}\to \infty$ we 	apply Lemma~\ref{lem:technical2} twice and compare: first with $H(z)=e^{\lambda'(e^z-z-1)}$, $H_1\equiv 1$ and $H_2\equiv 0$, and then with  $H(z)=e^{\lambda(e^z-z-1)}$, $H_1(z)=e^{(\lambda-\lambda')(e^z-z-1)}$ and $H_2\equiv 0$.
\end{itemize}
\subsection{Proof of Propositions~\ref{prop:omega}--\ref{prop:kac2}}
The moment generating function of $\omega(n)-\log \log x$, where $n$ is chosen uniformly at random from $[1,x]\cap\ZZ$, is
\[	M(z):=\sum_{k \ge 0}\frac{z^k}{k!} \EE_{n \le x} (\omega(n)-\log \log x)^k=\EE_{n \le x} e^{z(\omega(n)-\log \log x)}.\]
For $\Shift\in \RR$, we introduce
\begin{equation}\label{eq:gens2}
	M_{\Shift}(z):=\sum_{k \ge 0}\frac{z^k}{k!} \EE_{n \le x} (\omega(n)-\log \log x-\Shift)^k=\EE_{n \le x} e^{z(\omega(n)-\log \log x-\Shift)}=e^{-z\Shift}M(z).
\end{equation}
We shall need the following result of Selberg \cite[Thm.~2]{Selberg}: uniformly for $|s| \le B$ and $x \ge 3$ we have
\begin{equation}\label{eq:selberg}
	\EE_{n \le x} s^{\omega(n)} = (\log x)^{s-1}(F(s) + O_B( 1/\log x))
\end{equation}
for the $F$ in \eqref{eq:Fdef}. Substituting $s=e^z$ in \eqref{eq:selberg}, we obtain that
\begin{equation}\label{eq:MMA}
	M_{\Shift}(z) = e^{-z\Shift}(\log x)^{e^z-z-1}(F(e^z) + G(z))
\end{equation}
where $G(z)= O_C(1/\log x)$ for $|z|\le C$. First, we show that $G(z)$ makes a negligible contribution to the moments, by proving a version of \eqref{eq:D} with an explicit dependence on $k$.
\begin{cor}\label{cor:effectdelange}
Fix $A>0$. If  $1\le k\le A \log \log x$ and $-A \le \Shift \le A$ then
\[		\EE_{n \le x}\bigg(\frac{\omega(n)-\log \log x-\Shift}{\sqrt{\log \log x}} \bigg)^k =\frac{k!A_{k,\Shift}(x)}{(\log \log x)^{k/2}}+O_A\bigg(\frac{\AbsMom_k}{\log x}\exp\bigg(C_A \frac{k^{3/2}}{\sqrt{\log \log x}}\bigg)\bigg)\]
holds where $A_{k,\Shift}(x)$ is the coefficient of $z^k$ in $e^{-z\Shift}F(e^z)(\log x)^{e^z-z-1}$, and the implied constant and $C_A$ depend only on $A$.
\end{cor} 
\begin{proof}
Comparing coefficients in \eqref{eq:gens2} and \eqref{eq:MMA}, we have 
\[		\EE_{n \le x}\bigg(\frac{\omega(n)-\log \log x-\Shift}{\sqrt{\log \log x}} \bigg)^k =\frac{k!A_{k,\Shift}(x)}{(\log \log x)^{k/2}}+k!b_k\]
where $b_k$ is the coefficient of $z^k$ in $e^{-z\Shift}(\log x)^{e^z-z-1}G(z)$, divided by $(\log \log x)^{k/2}$. It remains to bound $k!b_k$. This is done by applying Lemma~\ref{lem:tech0} with $\lambda=\log \log x$ and $H(z) = e^{-z\Shift}(\log x)^{e^z-z-1}G(z)$, and using the estimates $G(z)=O_A(1/\log x)$ and $e^z-(z^2/2+z+1)=O_A(|z|^3)$ which hold for $|z|\le \sqrt{A}$.
\end{proof}
Since the error term in Corollary~\ref{cor:effectdelange} is negligible compared to the error terms in Proposition~\ref{prop:omega}--\ref{prop:kac2}, it remains to estimate $k!$ times $A_{k,\Shift}(x)$, the coefficient of $z^k$ in $e^{-z\Shift}F(e^z)(\log x)^{e^z-z-1}$.
\begin{itemize}
	\item To prove Proposition~\ref{prop:omega} we apply Corollary~\ref{cor:tech} with $\lambda=\log \log x$, $H(z)=e^{-z\Shift}F(e^z)(\log x)^{e^z-z-1}$ and $m=1$ (if $k$ is even) or $m=\min\{3,k\}$ (if $k$ is odd).
	In the notation of Corollary~\ref{cor:tech}, $b_0=1$, $b_1=F'(1)-\Shift=\MMc-\Shift$ and $b_3=\lambda/6 + O_A(1)$.
	\item To prove Proposition~\ref{prop:kac} we apply Lemma~\ref{lem:tech1} with $\lambda=\log \log x$ and $H(z)=e^{-z\Shift}F(e^z)(\log x)^{e^z-z-1}$.
\item To prove Proposition~\ref{prop:kac2} we apply Lemma~\ref{lem:technical2} with $\lambda=\log \log x$, $H(z)=e^{-z\Shift}F(e^z)(\log x)^{e^z-z-1}$, $H_1(z)=e^{-z\Shift}F(e^z)$ and $H_2\equiv 0$. 
\end{itemize}
\subsection{Proof of Proposition~\ref{prop:pb}}
We use the notation of Proposition~\ref{prop:pb}, and in particular $\lambda=\sum_{i=1}^{n}p_i$. The moment generating function of $Y-\lambda$ is
\[ M_{Y}(z):=\sum_{k \ge 0}\frac{z^k}{k!} \EE (Y-\lambda)^k=\EE e^{z (Y-\lambda)}=\prod_{i=1}^{n} \EE e^{z(Y_i-p_i)} = e^{-\lambda z} \prod_{i=1}^{n} (1+p_i(e^z-1)).\]
To prove the first part of the proposition we apply Corollary~\ref{cor:tech} with $H=M_Y$ and $m=1$ (if $k$ is even) or $m=3$ (if $k$ is odd). 
In the notation of Corollary~\ref{cor:tech}, $b_0=1$, $b_1=0$, $b_2 =-\Lambda/2$ and $b_3=\lambda/6 + O(\Lambda)$.
To prove the second part of the proposition we apply Lemma~\ref{lem:tech1} with $H=M_Y$. To prove the last part we apply Lemma~\ref{lem:technical2} with $H=M_Y$, $H_1(z) =  \prod_{i=1}^{n} (1+p_i(e^z-1))/e^{p_i(e^z-1)}$ and $H_2 \equiv 0$.
\section{Proofs of Generating function lemmas}\label{sec:lemmasproofs}
\subsection{Proof of Lemma \ref{lem:tech0}}
For a $2\pi$-periodic function $f$ we have
\begin{equation}\label{eq:period}
\int_{-\pi}^{\pi} f(\theta)\diff\theta = \int_{-\pi/2}^{\pi/2} (f(\theta)+f(\theta+\pi))\diff\theta.
\end{equation}
Recall $r=\sqrt{k/\lambda}$. Stirling's approximation shows that
\begin{equation}\label{eq:stir}
	\frac{e^{k/2}}{2\pi r^k} = \frac{\lambda^{k/2}\sqrt{\pi k}}{2\pi 2^{k/2} \Gamma\left( \frac{k}{2}+1\right)}(1+O(1/k)) = \frac{\lambda^{k/2}\AbsMom_k}{k!}\frac{1+O(1/k)}{2\sqrt{\pi/k}}.
\end{equation}
We need to show that the coefficient of $z^k$ in $k! \exp( \lambda z^2/2)H_1(z)$ is $\ll \lambda^{k/2} \AbsMom_k  \max_{|z|=r}|H_1(z)|$. This coefficient can be written as
\begin{equation}\label{eq:rep}
\begin{split}
\frac{k!}{2\pi r^k} &\int_{-\pi}^{\pi}\exp( \lambda r^2 e^{2i\theta}/2) H_1(re^{i\theta})e^{-i\theta k} \diff\theta \\
&=  \frac{k!\exp( \lambda r^2/2)}{2\pi r^k} \int_{-\pi/2}^{\pi/2} \exp( \lambda r^2(e^{2i\theta}-1)/2)(H_1(re^{i\theta})+(-1)^k H_1(-re^{i\theta}))e^{-i\theta k}\diff\theta
\end{split}
\end{equation}
by \eqref{eq:period}. Recall $\lambda r^2 = k$. By the triangle inequality and \eqref{eq:stir}, the right-hand side of \eqref{eq:rep} is
\begin{equation}\label{eq:rep2}
\begin{split}
& \ll \frac{k! e^{k/2}}{2\pi r^k} \max_{|z|=r} |H_1(z)+(-1)^k H_1(-z)|  \int_{-\pi/2}^{\pi/2} \exp( \Re (k(e^{2i\theta}-1)/2))\diff\theta\\
& \ll \lambda^{k/2} \AbsMom_k k^{1/2}\max_{|z|=r} |H_1(z)+(-1)^k H_1(-z)|  \int_{-\pi/2}^{\pi/2} \exp( \Re (k(e^{2i\theta}-1)/2))\diff\theta.
\end{split}
\end{equation} Uniformly for $\theta \in [-\pi/2,\pi/2]$ we have 
\begin{equation}\label{eq:cosineq}
	\Re (e^{2i\theta}-1) = \cos(2\theta)-1 \le -c\theta^2
\end{equation}
for some absolute constant $c>0$ (we may take $c=8/\pi^2$) and so the right-hand side of \eqref{eq:rep2} is
\begin{align*}
&\ll \lambda^{k/2} \AbsMom_k k^{1/2}\max_{|z|=r} |H_1(z)+(-1)^k H_1(-z)|  \int_{-\pi}^{\pi}e^{ -ck \theta^2} \diff\theta \\
&= \lambda^{k/2} \AbsMom_k k^{1/2}\max_{|z|=r} |H_1(z)+(-1)^k H_1(-z)| (2ck)^{-1/2} \int_{-\pi\sqrt{2ck}}^{\pi\sqrt{2ck}} e^{-t^2/2} \diff t\\
& \ll  \lambda^{k/2} \AbsMom_k \max_{|z|=r} |H_1(z)+(-1)^k H_1(-z)| 
\end{align*}
as needed, where in the equality we substituted $\theta=t/\sqrt{2ck}$. 
\subsection{Proof of Corollary \ref{cor:tech}}
We write
\[ k!a_k = k!(a_{k,1}+k!a_{k,2})\]
where $a_{k,1}$ (resp.~$a_{k,2}$) is the coefficient of $z^n$ in \[\exp(\lambda z^2/2)\sum_{j=0}^{m} b_j z^j \qquad (\text{resp.~}\exp(\lambda z^2/2)(H_1(z) -\sum_{j=0}^{m} b_j z^j)).\] Since $\exp(\lambda z^2/2)$ is the generating function of $(\lambda^{j/2} \mu_j/j!)_{j \ge 0}$, we may evaluate $a_{k,1}$ exactly as 
\[ a_{k,1} =\sum_{j=0}^{m} b_j \lambda^{(k-j)/2} \frac{\mu_{k-j}}{(k-j)!}.\]
If $m=k$ we are done since $a_{k,2}=0$ by construction. When $m\le k-1$, we appeal to Lemma~\ref{lem:tech0} in order to bound $k!a_{k,2}$. Precisely, we write $k! a_{k,2}$ as $k!/(k-m-1)!$ times $(k-m-1)!a_{k,2}$, and we bound $(k-m-1)!a_{k,2}$ by considering $a_{k,2}$ as the coefficient of $z^{k-m-1}$ in $\exp(\lambda z^2/2)$ times $(H_1(z) -\sum_{j=0}^{m} b_j z^j)/z^{m+1}$. The bound we obtain is
\[ k!a_{k,2}\ll \lambda^{\frac{k}{2}} \bigg(  \frac{k!}{(k-m-1)!}  \lambda^{-\frac{m+1}{2}} \AbsMom_{k-m-1} \max_{|z|=r} \big|\frac{H_1(z)-\sum_{j=0}^m  b_j z^j + (-1)^k(H_1(-z)-\sum_{j=0}^{m} b_j (-z)^j )}{z^{m+1}}\big|\bigg)\]
where $r=\sqrt{(k-m-1)/\lambda}$. We simplify the bound using the fact that $\max_{|z|=r}|\tfrac{F(z)-\sum_{i=0}^{m} F^{(i)}(0)z^i/i!}{z^{m+1}}|$ is bounded by $\max_{|z|=r} |\tfrac{F^{(m+1)}(z)}{(m+1)!}|$ for any power series $F$ converging absolutely in $|z|=r$.
\subsection{Proof of Lemma~\ref{lem:tech1}}
Recall $r=\sqrt{k/\lambda}$. We express $k!a_k$ using \eqref{eq:period} as
\begin{equation}\label{eq:akR1R2}
\begin{split}
k!a_k &=  \frac{k!}{2\pi r^{k}} \int_{-\pi}^{\pi} H_1(re^{i\theta})e^{-i\theta k} \diff\theta\\
&= \frac{k!}{2\pi r^{k}} \int_{-\frac{\pi}{2}}^{\frac{\pi}{2}}  \exp(k e^{2i\theta}/2)  ( H_1(re^{i\theta}) + (-1)^k H_1(-r e^{i\theta})) e^{-i\theta k} \diff\theta\\
&= \frac{k! e^{k/2}}{2\pi r^{k}} ( R_1+R_2)
\end{split}
\end{equation}
where
\begin{align*}
R_1 &:= \int_{-\frac{\pi}{2}}^{\frac{\pi}{2}}  \exp(k (e^{2i\theta}-1)/2)  (H_1(re^{i\theta})-H_1(r)+(-1)^k (H_1(-re^{i\theta})-H_1(-r))) e^{-i\theta k}\diff\theta,\\
R_2 &:= (H_1(r)+(-1)^k H_1(-r))\int_{-\frac{\pi}{2}}^{\frac{\pi}{2}}  \exp(k(e^{2i\theta}-1)/2)   e^{-i\theta k}\diff\theta.
\end{align*}
We study $R_1$ by treating separately $\pi/2 \ge |\theta| \ge k^{-1/3}$  and $|\theta| \le k^{-1/3}$. To bound the contribution of large $|\theta|$ we use \eqref{eq:cosineq}:
\begin{align*}
\int_{\pi/2 \ge |\theta| \ge k^{-1/3}}  &\exp(k (e^{2i\theta}-1)/2)  (H_1(re^{i\theta})-H_1(r)+(-1)^k(H_1(-re^{i\theta})-H_1(-r))) e^{-i\theta k}\diff\theta  \\
&\ll ( \max_{|z|=r} |H_1(z)-H_1(r)|+\max_{|z|=r} |H_1(z)-H_1(-r)|) \int_{\pi/2 \ge |\theta| \ge k^{-1/3}}  e^{-ck \theta^2} \diff\theta\\
& \ll r \max_{|z|=r} |H_1'(z)| e^{-ck^{1/3}}.
\end{align*}
To study small $\theta$ we Taylor-expand $e^{2i\theta}-1$ as $2i\theta -2\theta^2 -4i\theta^3/3 + O(\theta^4)$ to find
\[ \exp(k(e^{2i\theta}-1)/2 - i\theta k )= e^{-k\theta^2} (1 -2ik\theta^3/3 + O(k^2\theta^6 + k\theta^4))\]
for $|\theta| \le k^{-1/3}$. We also expand $H_1(\pm re^{i\theta})-H_1(\pm r)$ as
\[ H_1(\pm re^{i\theta})-H_1(\pm r) = \pm ir H_1'(\pm r)\theta + O(E\theta^2)\]
where $E$ is defined in \eqref{eq:Edef}. It follows that
\begin{align*}
	&\int_{|\theta| \le k^{-\frac{1}{3}}}  \exp(k (e^{2i\theta}-1)/2)  (H_1(re^{i\theta})-H_1(r)+(-1)^k(H_1(-re^{i\theta})-H_1(-r))) e^{-i\theta k}\diff\theta  \\
	&= \int_{|\theta|\le k^{-\frac{1}{3}}} e^{-k\theta^2} ( 1-2ik\theta^3/3 + O(k^2 \theta^6 + k\theta^4)) \theta ( ir H_1'(r)- (-1)^kir H_1'(-r) + O(E |\theta|)) \diff\theta\\
	&= \frac{1}{2k}\int_{|t|\le \sqrt{2}k^{\frac{1}{6}}} e^{-\frac{t^2}{2}} ( 1-i(2k)^{-1/2} t^3/3 + O(t^6/k + t^4/k)) t ( ir H_1'(r)- (-1)^kir H_1'(-r) + O(E |t|k^{-1/2})) \diff t
\end{align*}
where in the last equality we substituted $\theta = t/\sqrt{2k}$. The integral $\int_{|t| \le A} t^je^{-t^2/2}\diff t$ vanishes for odd $j$ and is equal to $\sqrt{2\pi}\mu_j + O( e^{-c A^2})$ for $0 \le j \le 8$. Hence
\begin{align*}
	\int_{|\theta| \le k^{-\frac{1}{3}}}  &\exp(k (e^{2i\theta}-1)/2)  (H_1(re^{i\theta})-H_1(r)+(-1)^k(H_1(-re^{i\theta})-H_1(-r))) e^{-i\theta k}\diff\theta  \\
&=  \frac{r}{2k}(H_1'(r)-(-1)^k H_1'(-r))(\sqrt{2\pi} (2k)^{-1/2} + O(k^{-1}))+O(E k^{-3/2}). 
\end{align*}
We have shown that
\begin{equation}\label{eq:R1bnd}
R_1 =\frac{r}{2k}(H_1'(r)-(-1)^k H_1'(-r))(\sqrt{\pi/k}  + O(k^{-1}))+O(E k^{-3/2} + r\max_{|z|=r} |H'_1(z)|  e^{-ck^{1/3}})\ll E k^{-3/2}
\end{equation}
holds where the inequality follows from the definition of $E$. We now treat $R_2$. If $k$ is even, we can evaluate $R_2$ exactly. Indeed, because the $k$th coefficient of $e^{\lambda z^2/2}$ is $\lambda^{k/2} 2^{-k/2}/(k/2)! = \lambda^{k/2} \mu_k/k! = \lambda^{k/2} \AbsMom_k/k!$ when $k$ is even we see
\begin{multline*}
	\lambda^{k/2}\frac{\AbsMom_k}{k!}  =\frac{1}{2\pi r^k} \int_{-\pi}^{\pi} \exp(\lambda r^2 e^{2i\theta}/2) e^{-i\theta k} \diff\theta  =\frac{1}{\pi r^k} \int_{-\pi/2}^{\pi/2} \exp(k e^{2i\theta}/2) e^{-i\theta k} \diff\theta\\
	 =\frac{e^{k/2}}{\pi r^k} \int_{-\pi/2}^{\pi/2} \exp(k (e^{2i\theta}-1)/2) e^{-i\theta k} \diff\theta
\end{multline*}
by \eqref{eq:period}, so
\[ \frac{e^{k/2}}{2\pi r^k} R_2 = \frac{H_1(r)+H_1(-r)}{2}\frac{\AbsMom_k}{k!}\lambda^{k/2}.\]
If $k$ is odd, we can estimate $R_2$ in the same way we studied $R_1$: by Taylor-expanding $k(e^{2i\theta}-1)/2$ (for $|\theta|\le k^{-1/3}$) and using \eqref{eq:cosineq} (for $|\theta|\ge k^{-1/3}$), resulting in
\[ R_2 = (H_1(r)-H_1(-r)) \sqrt{\frac{\pi}{k}}(1+O(1/k)).\]
From \eqref{eq:akR1R2}, \eqref{eq:R1bnd}, our estimates for $R_2$ and \eqref{eq:stir}, the proof is concluded.

\subsection{Proof of Lemma \ref{lem:technical2}}
We have, by assumption,
\begin{align*}
	a_k&= \frac{1}{2\pi r^{k}} \int_{-\pi}^{\pi} H(re^{i\theta}) e^{-ik\theta}\diff \theta\\
	&= \frac{1}{2\pi r^{k}} \int_{-\pi}^{\pi} \exp(\lambda(e^{re^{i\theta}}-re^{i\theta}-1)) e^{-i\theta k} ( H_1(re^{i\theta}) + H_2(re^{i\theta}))\diff\theta\\
	&=  \frac{\exp(\lambda(e^r-r-1))}{2\pi r^k} \int_{-\pi}^{\pi} e^{\lambda P_r(\theta)} e^{-i\theta k} ( H_1(re^{i\theta}) + H_2(re^{i\theta}))\diff\theta
\end{align*}
where \[P_r(\theta):= e^{r e^{i\theta}}-re^{i\theta}-1-(e^r -r -1)=\sum_{j\ge 2}\frac{r^j (e^{i\theta j}-1)}{j!}.\] We Taylor-expand $H_1$ as 
\[	H_1(re^{i\theta}) =H_1(r) + i\theta r H_1'(r) + O_A( E_1\theta^2),\qquad E_1:=r\max_{|z|= r}|H_1'(z)|+ r^2\max_{|z|= r} |H''_1(z)|.\] This expansion of $H_1$ shows that
\[ a_k = \frac{\exp(\lambda(e^{r}-r-1))}{2\pi r^k} (H_1(r) I_0 + i rH_1'(r)I_1 + O_A( E_1 J_1 + \max_{|z|= r}|H_2(z)| J_0))\]
holds, where 
\[ I_j:=  \int_{-\pi}^{\pi} e^{\lambda P_r(\theta)}e^{-i\theta k}\theta^{j} \diff\theta, \qquad  J_j :=  \int_{-\pi}^{\pi} e^{\lambda \Re P_r(\theta)} \theta^{2j} \diff\theta \]
for $j=0,1$. It remains to estimates $I_j, J_j$. We claim that uniformly for $|\theta|\le \pi$ we have
\begin{equation}\label{eq:reh}
	\Re P_r(\theta) \le -c\cdot  \begin{cases} r^2\theta^2 & \text{if }|\theta| \le 3\pi/4, \\ r^3 + r^2 ( |\theta|-\pi)^2& \text{if }\pi \ge |\theta| \ge 3\pi/4,\end{cases}
\end{equation}
for an absolute, positive constant $c>0$ ($c=1/30$ works, say). Indeed, \[\Re P_r(\theta)=\Re \sum_{j\ge 2}\frac{r^j (e^{i\theta j}-1)}{j!} = - \sum_{j\ge 2}\frac{r^j (1-\cos(\theta j))}{j!}\]
and by considering only $j=2$ and $j=3$ we obtain \eqref{eq:reh}. Using \eqref{eq:reh} we see that
\[ J_j \ll (\lambda r^2 )^{-\frac{1}{2}-j} + (\lambda r^2 )^{-\frac{1}{2}} \exp(-(c/2)\lambda r^3)\]
holds for $j=0,1$. Before studying $I_j$, 
recall that, as stated in the lemma, we take $r$ to be the positive solution to \[r(e^r-1)=\frac{k}{\lambda}.\]
This ensures that \[  P'_r(0) = i\frac{k}{\lambda}\]
holds. Recall that we assume $k/\lambda \le A(e^A-1)$, so that in particular $r\le A$. Let
\[ T:=\min\{\pi,(\lambda r^2 )^{-1/3},(\lambda r^2)^{-1/4}\}.\]
It is useful to note that $r \asymp_A \sqrt{k/\lambda}$ because $k/\lambda = r(e^r-1)\asymp_A r^2$ and in particular $T \asymp_A k^{-1/3}$. 
To study $I_j$ we consider separately $|\theta|\le T$ and $|\theta|\in [T,\pi]$. To bound the contribution of $|\theta| \in [T,\pi]$ we use \eqref{eq:reh} and the triangle inequality to find that
\begin{align*}
\bigg| \int_{|\theta| \in [T,\pi]} e^{\lambda P_r(\theta)}e^{-i\theta k}\theta^{j} \diff\theta\bigg|& \ll \int_{T}^{\pi}\exp(-c\lambda r^2 \theta^2)\theta^j\diff \theta  + (\lambda r^2 )^{-\frac{1}{2}}\exp(-c\lambda r^3)\\
& \ll (\lambda r^2)^{-\frac{1+j}{2}} \exp(-(c/2)\lambda r^2 T^2)+ (\lambda r^2 )^{-\frac{1}{2}}\exp(-c\lambda r^3)
\end{align*}
for $j=0,1$, where we used the standard estimate
\[ \int_{B}^{\infty} e^{-t^2/2}t^j\diff t \ll e^{-B^2/4}\]
for $j=0,1$ and $B \ge 0$. For $\theta \in [-T,T]$ we use the Taylor approximation
\[	P_r(\theta) - i\theta \frac{k}{\lambda} = r^2 \bigg( - \frac{ b_2}{2} \theta^2 + b_3 \theta^3 +b_4 \theta^4\bigg) \]
where
\begin{align*}
	b_{2} &= \sum_{j\ge 2} \frac{r^{j-2} j^{2}}{j!}= e^r + \frac{e^r-1}{r},\qquad  b_3 = -\frac{i}{6}\sum_{j\ge 2} \frac{r^{j-2} j^{3}}{j!},\\
	|b_4| &= \bigg|\sum_{j \ge 2} \frac{r^{j-2} (e^{i\theta j}-(1+i\theta j+(i\theta j)^2/2! + (i\theta j)^3/3!)) }{ \theta^4 j!} \bigg|\ll \sum_{j \ge 2} \frac{r^{j-2}j^4}{j!}.
\end{align*}
Note that $b_2\ge 2$ and $|b_2|,|b_3|,|b_4|=O_A(1)$.
We find that
\begin{align*}
	I_j &=\int_{-T}^{T} \exp(-\lambda b_2 r^2  \theta^2/2)( 1 + b_3 \theta^3 \lambda r^2  + O_A( \theta^6 \lambda^2 r^4  + \theta^4 \lambda r^2 ))\theta^j \diff \theta\\
	&\qquad +O( (\lambda r^2)^{-\frac{1+j}{2}} \exp(-(c/2)\lambda r^2 T^2)+ (\lambda r^2 )^{-\frac{1}{2}}\exp(-c \lambda r^3))
\end{align*}
for $j=0,1$. Since $\int_{-B}^{B} e^{-t^2/2}t^j \diff t=0$ for odd $j$, and $\int_{-B}^{B} e^{-t^2/2}t^j\diff t = \sqrt{2\pi}\mu_j +O(e^{-B^2/4})$ for $0 \le j \le 6$, we may simplify our estimate for $I_j$ as
\begin{align*}
	I_0&= \frac{1}{\sqrt{b_2 \lambda r^2 }} ( \sqrt{2\pi} + O_A( (\lambda r^2)^{-1} +\exp(-b_2 r^2\lambda T^2/4 ))) \\
	&\qquad+ O_A\left( \frac{\exp(-(c/2)\lambda r^2 T^2) + \exp(-c\lambda r^3)	}{\sqrt{b_2 \lambda r^2}} \right)\\
	&= \frac{\sqrt{2\pi}}{\sqrt{b_2\lambda r^2}}  (1 + O_A( 1/k  + \exp(-c_Ak^{3/2}\lambda^{-1/2})))
\end{align*}
when $j=0$ and 
\[ I_1 \ll_A \frac{1}{\sqrt{b_2\lambda r^2}} ( 1/k  + \exp(-c_Ak^{3/2}\lambda^{-1/2}))\]
when $j=1$. In summary, 
\[ a_k = \frac{\exp(\lambda(e^r - r - 1))}{\sqrt{2\pi b_2 \lambda r^2 }r^k}  (H_1(r)+O_A( (1 /k +  \exp(-c_Ak^{3/2}\lambda^{-1/2}))(E_1+ |H_1(r)|) +\max_{|z|= r}|H_2(z)|	)).\]
We simplify the main term using Stirling's formula applied to $\Gamma(k/2+1)$ as follows:
\[ 	\frac{\exp(\lambda(e^r - r - 1))}{\sqrt{2\pi b_2 \lambda r^2 }r^k} =  \frac{\lambda^{k/2}}{2^{k/2}\Gamma(k/2 + 1)}\frac{\exp(\lambda(e^r-r-1))}{\sqrt{2b_2 (\lambda r^2 /k)}} (k/(e\lambda r^2))^{k/2}(1+O(1/k)).\]
Letting $t:=\lambda r^2 /k = r/(e^r-1)$, we can write
\[	\exp(\lambda(e^r-r-1))(k/(e\lambda r^2))^{k/2} = \exp\bigg( \frac{k}{2} \bigg( 2t \frac{e^r-\frac{r^2}{2}-r-1}{r^2} +t-1 - \log t\bigg)\bigg),\]
giving the result.
\section{Origin of transition at \texorpdfstring{$k\asymp (\log \log x)^{1/3}$}{k ~ log log x  1/3}}\label{sec:intuition}
Let $\lambda>0$ and $X(\lambda) \sim \mathrm{Poisson}(\lambda)$. We explain intuitively the transition in the $k$th moment of $X(\lambda)-\lambda$ (resp.~$\omega(n)-\log \log x$) once $k\asymp \lambda^{1/3}$ (resp.~$k\asymp (\log \log x)^{1/3}$). We begin by stating three informal claims on $G\sim N(0,1)$ and $(X(\lambda)-\lambda)/\sqrt{\lambda}$.
\begin{claim}\label{claim:ga} 
We have $\EE (G^{2k}) \sim \EE (G^{2k} \mathbf{1}_{G=O(\sqrt{k})})$. This is false if $O(\sqrt{k})$ is replaced with $o(\sqrt{k})$. 
\end{claim}
\begin{claim}\label{claim:po} Let $\lambda > 0$. When $k=O(\lambda)$ we have
	 \[\EE\bigg( \bigg( \frac{X(\lambda)-\lambda}{\sqrt{\lambda}}\bigg)^{2k} \mathbf{1}_{\big|\frac{X(\lambda)-\lambda}{\sqrt{\lambda}}\big| =O(\sqrt{k})} \bigg) \sim  \EE\bigg( \bigg( \frac{X(\lambda)-\lambda}{\sqrt{\lambda}}\bigg)^{2k}\bigg).\]
	 This is false if $O(\sqrt{k})$ is replaced with $o(\sqrt{k})$. 
\end{claim}
\begin{claim}\label{claim:stirling}
For $j \in \ZZ$ and $\lambda >0$ let $I=[j,j+1)$ and $R = j/\sqrt{\lambda}$. When $j=O(\sqrt{\lambda})$ we have
	\begin{align}\label{eq:stirling1}
		\PP(G \in I)&\approx e^{-\lambda R^2/2},\\
		\label{eq:stirling2}	\PP((X(\lambda)-\lambda)/\sqrt{\lambda} \in I)&\approx  e^{-\lambda ((1+R)\log(1+R)-R)}		
	\end{align}
	The probabilities \eqref{eq:stirling1} and \eqref{eq:stirling2} are asymptotic to one another if and only if $\lambda R^3=o(1)$, i.e.~if $j=o(\lambda^{1/6})$. Once $j\gg \lambda^{1/6}$, \eqref{eq:stirling1} is \textit{much larger} than \eqref{eq:stirling2}.
\end{claim}
We now combine the claims. If $k=o(\lambda^{1/3})$ then the $2k$th moment of $(X(\lambda)-\lambda)/\sqrt{\lambda}$ is supported on the event $(X(\lambda)-\lambda)/\sqrt{\lambda} =o(\lambda^{1/6})$ by Informal Claim \ref{claim:po}. When $j=o(\lambda^{1/6})$, the probability that $(X(\lambda)-\lambda)/\sqrt{\lambda} \in [j,j+1)$ grows like the probability that $G\in [j,j+1)$, by \eqref{eq:stirling1}-\eqref{eq:stirling2}, so by Informal Claim \ref{claim:ga} the $2k$th Poisson moment must grow like the $2k$th Gaussian moment. However, once $k\gg \lambda^{1/3}$, the $2k$th moment has significant contribution from the event that $(X(\lambda)-\lambda)/\sqrt{\lambda} \asymp \lambda^{1/6}$, whose probability is much larger than the probability that $G \asymp \lambda^{1/6}$. The entirety of the discussion can be adapted to $\omega$, since Sath\'e \cite{Sathe} proved that
\[\PP_{n \le x}(\omega(n)-1=k) \sim \PP(X(\log \log x)=k)\]
holds for $k\sim \log \log x$ and 
\[\PP_{n \le x}(\omega(n)-1=k) \asymp \PP(X(\log \log x)=k)\]
holds for $k=O(\log \log x)$. We conclude that the $2k$th moment of $(\omega(n)-\log \log x)/\sqrt{\log \log x}$ deviates from  Gaussian once $k\gg (\log \log x)^{1/3}$ due to the density of integers with $(\omega(n) - \log \log x)/\sqrt{\log \log x} \asymp (\log \log x)^{1/6}$ being \textit{much larger} than the probability that $G \asymp (\log \log x)^{1/6}$.
These are atypical integers, but they contribute significantly to the $2k$th moment once $k\gg (\log \log x)^{1/3}$.

It remains to justify the claims.
For Informal Claim \ref{claim:ga}, write
\[ \EE (G^{2k} \mathbf{1}_{|G|\ge T})=
\frac{1}{\sqrt{2\pi}}\int_{|t|\ge T} t^{2k}e^{-t^2/2}\diff{t}\]
and observe that the even function $t\mapsto t^{2k} e^{-t^2/2}$ increases for $0\le t \le \sqrt{2k}$ and decreases (rapidly) for $t \ge \sqrt{2k}$. For Informal Claim \ref{claim:po}, recall we have shown $\EE (G^{2k}) \asymp \EE (((X(\lambda)-\lambda)/\sqrt{\lambda})^{2k})$, and use Stirling's approximation to see
\begin{multline*} \EE\bigg( \bigg( \frac{X(\lambda)-\lambda}{\sqrt{\lambda}}\bigg)^{2k} \mathbf{1}_{\big|\frac{X(\lambda)-\lambda}{\sqrt{\lambda}}\big| > T} \bigg) \approx e^{-\lambda} \sum_{|j|>T\sqrt{\lambda}} (j/\sqrt{\lambda})^{2k} \frac{\lambda^{\lambda+j}}{(\lambda+j)!} \\
	\approx \frac{1}{\sqrt{2\pi\lambda}}\sum_{|j|> T}(j/\sqrt{\lambda})^{2k} e^{-\frac{j^2}{2\lambda}} \approx \frac{1}{\sqrt{2\pi}} \int_{|t|\ge T} t^{2k}e^{t^2/2}\diff{t}
\end{multline*}
where $j \in \ZZ$ stands for values of $X(\lambda)-\lambda$. We turn to Informal Claim \ref{claim:stirling}. The last part of the claim follows by comparing \eqref{eq:stirling1} with \eqref{eq:stirling2}. To justify \eqref{eq:stirling1} we use
\[ \PP(G \in I) = \frac{1}{\sqrt{2\pi}}\int_{j}^{j+1} e^{-t^2/2}\diff{t} \approx e^{-j^2/2} = e^{-\lambda R^2/2}. \]
For \eqref{eq:stirling2} we use Stirling's approximation:
\[ 		\PP((X(\lambda)-\lambda)/\sqrt{\lambda} \in I)=  e^{-\lambda} \sum_{k \in [\lambda+\sqrt{\lambda}j,\lambda+\sqrt{\lambda}(j+1))} \frac{\lambda^{k}}{k!} \approx e^{-\lambda} \frac{\lambda^{\lambda+\sqrt{\lambda} j}}{(\lambda+\sqrt{\lambda}j)!}\approx  e^{-\lambda ((1+R)\log(1+R)-R)}.\]

\bibliographystyle{abbrv}
\bibliography{references}

\Addresses
\end{document}